\documentclass[12pt]{amsart}

\usepackage{amssymb}
\usepackage{amsmath}
\usepackage{amsthm}
\usepackage{amscd}
\usepackage[margin=1.025in]{geometry}%
\usepackage{hyperref}

\theoremstyle{plain}
\newtheorem{theorem}{Theorem}[section]

\newtheorem{corollary}{Corollary}[section]
\newtheorem{lemma}{Lemma}[section]

\theoremstyle{remark}
\newtheorem{definition}{Definition}[section]
\newtheorem{remark}{Remark}[section]
\newtheorem{examples}{Examples}[section]

\DeclareMathOperator*{\esssup}{ess\,sup}
\DeclareMathOperator{\dist}{dist}
\DeclareMathOperator{\supp}{supp}

\DeclareMathOperator{\clos}{clos}
\DeclareMathOperator{\lin}{span}

\newcommand{\round}[1]{{\ooalign{\hfil\raise .10ex\hbox{\scriptsize#1}\hfil\crcr\mathhexbox20D}}}

\newcommand{\ff}{\mathcal{F}}

\newcommand{\eng}{\mathcal{E}}

\newcommand{\cal}[1]{\ensuremath{\mathcal{#1}}}

\newcommand{\loc}{\text{loc}}

\title{Metrics and spectral triples for Dirichlet and resistance forms}
\author[Hinz]{Michael Hinz$^{1,2}$}
\address{Department of Mathematics, Bielefeld University, Postfach 100131, 33501 Bielefeld, Germany}
\email{mhinz@math.uni-bielefeld.de}
\thanks{$^1$Research supported in part by the Alexander von Humboldt Foundation Feodor (Lynen Research Fellowship Program)}
\author[Kelleher]{Daniel J. Kelleher$^2$}
\address{Department of Mathematics, University of Connecticut, Storrs, CT 06269-3009 USA}
\email{daniel.kelleher@uconn.edu}
\thanks{$^2$Research supported in part by NSF grant DMS-0505622}
\author[Teplyaev]{Alexander Teplyaev$^2$}
\address{Department of Mathematics, University of Connecticut, Storrs, CT 06269-3009 USA}
\email{alexander.teplyaev@uconn.edu}

\begin{document}

\begin{abstract}
The article deals with intrinsic metrics, Dirac operators and spectral triples induced by regular Dirichlet and resistance forms. 
We show, in particular, that if a local resistance form is given and the space is compact in resistance metric, then the intrinsic metric yields a geodesic space. Given a regular Dirichlet form, we consider Dirac operators within the framework of differential $1$-forms proposed by Cipriani and Sauvageot, and comment on its spectral properties. 
If the Dirichlet form admits a carr\'e operator and the generator has discrete spectrum, then we can construct a related spectral triple, and in the compact and strongly local case the associated Connes distance coincides with the intrinsic metric. We finally give a description of the intrinsic metric in terms of vector fields.
\tableofcontents
\end{abstract}
\maketitle

\section{Introduction}

In this article we study intrinsic metrics, Dirac operators and spectral triples associated with Dirichlet and resistance forms. A regular symmetric Dirichlet form on a locally compact space $X$ allows the localization of energy by means of energy measures $\Gamma(f)$, $f\in\mathcal{F}$, in the sense of Fukushima \cite{FOT94} and LeJan \cite{LJ78}. These energy measures may or may not be absolutely continuous with respect to the given reference measure, but it is always possible to find measures $m$ that are energy dominant, i.e. such that the energy measure $\Gamma(f)$ of every function $f\in\mathcal{F}$ is absolutely continuous with respect to $m$. 
If the given Dirichlet form is strongly local, \cite[Section 3.2]{FOT94}, then for any energy dominant measure $m$ we can consider an intrinsic distance $d_{\Gamma,m}$ that generalizes the classical expression
\[d(x,y)=\sup\left\lbrace f(x)-f(y)\ | \ f\in C^1 \text{ such that } |\nabla f|\leq 1\right\rbrace.\]
The definition of $d_{\Gamma,m}$ depends on the choice of $m$, and in general different measures $m$ will lead to different metrics. The first papers that studied intrinsic metrics induced by strongly local regular Dirichlet forms on locally compact spaces were \cite{BM95, Stu94, Stu95}, 
and 
more recent references are \cite{BdMLSt, KZh, Sto10}. Intrinsic metrics for non-local forms have been studied for the first time in \cite{FLW}, where in particular a Rademacher type theorem is proved for general regular Dirichlet forms, \cite[Theorem 4.9]{FLW}. Intrinsic metrics in infinite dimensional situations are considered in \cite{Hi03}.
A typical question is whether a locally compact space $X$ equipped with the intrinsic metric coming from a strongly local Dirichlet form is geodesic or at least a length space, i.e. such that the intrinsic metric coincides with the shortest path metric. Some of the corresponding results of Sturm \cite{Stu94, Stu95} have later been simplified by Stollmann, \cite{Sto10}. In these references it is assumed that the original reference measure is energy dominant and that the topology induced by the intrinsic metric coincides with the original topology. Under these assumptions the space, equipped with the intrinsic metric, is a length space, \cite{Sto10}. The same arguments allow to prove this result also for intrinsic metrics $d_{\Gamma,m}$ with respect to an arbitrary energy dominant measure $m$. The question whether or not the topology induced by $d_{\Gamma,m}$ coincides with the original topology on $X$ is known to be characterized by a compact embedding of a ball of Lipschitz functions into the space of continuous functions, see Theorems \ref{T:general} and \ref{T:converse} below. This result is not new, in more abstract context it has been shown by Rieffel \cite{R98}, Pavlovi\'c \cite{Pa98} and Latr\'emoli\`ere \cite{La07,Lat}. For expository reasons we quote a version in the language of Dirichlet forms.

The second setup we investigate is that of resistance forms $(\mathcal{E},\mathcal{F})$ in the sense of Kigami \cite{Kig01, Kig03, Kig12}. One of the most prominent examples for a resistance form is the standard energy form on the Sierpinski gasket, see for instance \cite{Kig93, Kig01}. Neither a topology nor a measure are needed to define a resistance form on a set $X$, and every resistance form determines a metric $d_R$ on $X$, the so-called resistance metric. A resistance form gives rise to a Dirichlet form in the sense of \cite{FOT94} if a suitable reference measure $m$ is specified, and under some conditions the resulting form will be regular (with respect to the topology induced by $d_R$). If $(X, d_R)$ is compact, the Dirichlet form will always be regular. In this case we may proceed as before and consider the intrinsic metric $d_{\Gamma, m}$. It turns out that in the local case the space $(X, d_{\Gamma,m})$ is always a length space. To our knowledge this result is new. Since compactness in $d_R$ implies compactness in $d_{\Gamma,m}$ any such space is even geodesic, i.e. any two points $x$ and $y$ can be joined by a path of length $d_{\Gamma,m}(x,y)$. We also give an example for a space $X$ that is compact in intrinsic 
metric$d_{\Gamma,m}$, but non-compact in the resistance metric~$d_R$. A~very special situation arises for resistance forms on dendrites (topological trees), \cite{Kig95}. If the dendrite is compact in resistance metric, then any intrinsic metric will itself be a resistance metric.

Another question we are interested in is the existence of Dirac operators and spectral triples associated with Dirichlet forms. In noncommutative geometry spectral triples are used to encode geometric information \cite{C94, GVF}. Recently several authors have begun to discuss spectral triples for fractals, \cite{CIL, CIS, GI03}. It would be interesting to see how these objects are related to recent research in mathematical physics on fractals, see for instance \cite{ADT1, ADT2}. In \cite{CGIS12} the authors investigate spectral triples for the Sierpinski gasket. They consider a parametrized family of spectral triples associated with non-local operators, first on circles and later on the gasket. The singularity of the energy with respect to the self-similar Hausdorff measure is described in terms of an energy dimension, defined by a trace formula, that differers from the Hausdorff dimension. In \cite{PB} spectral triples on ultrametric spaces are constructed. There the authors represent the ultrametric space in terms of a uniquely associated tree and base their construction on a notion of choices. Later on, they use Dirac operators, traces and integration over the space of choices to define associated Dirichlet forms. Further results related to spectral triples can be found in \cite{BMR}, where dynamical systems are studied and \cite{PaB}, where tiling spaces are investigated.

At first, 
we take a rather abstract point of view upon the existence of Dirac operators and spectral triples, as proposed by Cipriani and Sauvageot in \cite{CS03, CS09}. Given a regular symmetric Dirichlet form $(\mathcal{E},\mathcal{F})$ we follow this approach to a first order calculus  and consider a Hilbert space $\mathcal{H}$ of $L_2$-differential $1$-forms and a first order derivation $\partial$ associated with  $(\mathcal{E},\mathcal{F})$. Typical examples arise from Dirichlet forms on fractals such as Sierpinski carpets or gaskets, others from (relativistic) Schr\"odinger operators or Schr\"odinger operators associated with L\'evy processes. 
We recall the definition of a related Dirac operator $\mathbb{D}$ from \cite{HTb}, and for the special case that the generator $L$ of $(\mathcal{E},\mathcal{F})$ has pure point spectrum we describe the spectral representation for $\mathbb{D}$. 

Our interest in pure point spectrum arises, in particular, 
from its appearance in the study  of Laplacians  
on finitely ramified symmetric fractals and related graphs, see \cite{BK,MT,Sab,T98}. These studies are influenced by the mathematical theory of Anderson localization for random Hamiltonians, 
see \cite{AM,CL,SW} and references therein. It is an interesting question, which we do not address, whether  our results can be related to the pure point spectrum of random orthogonal
polynomials on the circle (see \cite{T3,GT,S-OPUC} and references therein) and the quasi-circles considered in \cite{C94}.

To consider spectral triples we partially follow the definition used in \cite{CGIS12}, it differs slightly from the classical one, cf. \cite{GVF}. This seems reasonable, because we are particularly interested in spectral triples on fractal spaces, and the latter may roughly speaking have an infinite dimensional first cohomology. A precise statement can be found in \cite[Theorem 3.9]{CGIS13}, see also \cite{IRT}.
If the original reference measure itself is energy dominant for $(\mathcal{E},\mathcal{F})$ and the generator $L$ has discrete spectrum, then we can verify the existence of a spectral triple $(\mathbb{A},\mathbb{H},\mathbb{D})$ for the $C^\ast$-algebra $\mathbb{A}$ obtained as the uniform closure of the collection $\mathbb{A}_0$ of continuous compactly supported functions of finite energy that have essentially bounded energy densities. See formula (\ref{E:A}) and Theorem \ref{T:triple}. To verify this result we make use of a direct integral representation for $\mathcal{H}$ from \cite{HRT}, stated in Theorem \ref{T:measurablefield}. In the strongly local case and under some conditions (for instance if $X$ is compact), a related Connes distance $d_{\mathbb{D}}$ on $X$ coincides with the intrinsic metric $d_{\Gamma,m}$, see Theorem \ref{T:Connesmetric}.

We would like to point out that our spectral triples differ from the spectral triples on Cantor sets considered in \cite{PB}. More precisely, the authors in \cite{PB} first construct a spectral triple and then derive a regular Dirichlet form from this spectral triple. If our method is applied to this regular Dirichlet form then one obtains a spectral triple that is different from the one they started with.

As a last item, we provide  a description of the intrinsic metric $d_{\Gamma,m}$ in terms of bounded vector fields. 
This resembles the situation in sub-Riemannian geometry. 
Our interest in vector fields arises, in particular, 
from the study of gradients on self-similar fractals \cite{DRS,Hino10,PT,Str,T00} and, 
more generally, on  fractafolds (see \cite{St03,ST}). 

Section \ref{S:Dirichlet} is concerned with intrinsic metrics for regular Dirichlet forms, and Section \ref{S:resistance} with the resistance form case. Dirac operators and spectral triples are investigated in Sections \ref{S:Dirac} and \ref{S:triples}, respectively. Section \ref{S:gradient} rephrases the definition of the intrinsic metric in terms of vector fields. Two straightforward facts about composition and multiplication of energy finite functions are stated in a short appendix.

When dealing with symmetric bilinear or conjugate symmetric sesquilinear expressions $(f,g)\mapsto Q(f,g)$ we write  $Q(f):=Q(f,f)$ to shorten notation.

\subsection*{Acknowledgments } 
{The authors are very  grateful to Jean Bellissard  for  important  and helpful discussions leading to this paper. They also thank Joe P. Chen, Naotaka Kajino, Daniel Lenz and the anonymous referees for valuable suggestions.}

\section{Length spaces for local regular Dirichlet forms}\label{S:Dirichlet}

Let $(X,d)$ be a locally compact separable metric space and $\mu$ a nonnegative Radon measure on $X$ such that $\mu(U)>0$ for any nonempty open set $U\subset X$. A pair  $(\mathcal{E},\mathcal{F})$ is called a symmetric \emph{Dirichlet form} on $L_2(X,\mu)$ if it satisfies the following conditions:
\begin{enumerate}
\item[(DF1)] $\mathcal{E}:\mathcal{F}\times\mathcal{F}\to \mathbb{R}$ is a nonnegative definite bilinear form on a dense subspace $\mathcal{F}$ of $L_2(X,\mu)$,
\item[(DF2)] $(\mathcal{E},\mathcal{F})$ is closed, i.e. $(\mathcal{F},\mathcal{E}_1)$, where $\mathcal{E}_1(f,g):=\mathcal{E}(f,g)+\left\langle f,g\right\rangle_{L_2(X,\mu)}$, is a Hilbert space,
\item[(DF3)] $(\mathcal{E},\mathcal{F})$ has the \emph{Markov property}, i.e. $u\in\mathcal{F}$ implies $(0\vee u)\wedge 1\in\mathcal{F}$ and 
\[\mathcal{E}((0\vee u)\wedge 1)\leq \mathcal{E}(u).\]
\end{enumerate}
A Dirichlet form $(\mathcal{E},\mathcal{F})$ is called \emph{regular} if in addition 
\begin{enumerate}
\item[(DF4)] the space $\mathcal{C}:=C_c(X)\cap \mathcal{F}$ is both uniformly dense in the space of compactly supported continuous functions $C_c(X)$ and dense in $\mathcal{F}$ with respect to the Hilbert space norm $f\mapsto\mathcal{E}_1(f)^{1/2}$.
\end{enumerate}
See for instance \cite{FOT94, MR92}. Now let $(\mathcal{E},\mathcal{F})$ be a regular Dirichlet form on $L_2(X,\mu)$. Since 
\[\mathcal{E}(fg)^{1/2}\leq \left\|f\right\|_{L_\infty(X,\mu)}\mathcal{E}(g)^{1/2}+\left\|g\right\|_{L_\infty(X,\mu)}\mathcal{E}(g)^{1/2},\ \ f,g\in\mathcal{C},\]
cf. \cite[Corollary I.3.3.2]{BH91}, the space $\mathcal{C}$ is an algebra of bounded functions, usually referred to as the \emph{Dirichlet algebra}. From a representation theoretic point of view it has been studied in detail in \cite{C06}. For any $f\in\mathcal{C}$ we may define a nonnegative Radon measure $\Gamma(f)$ on $X$ by  
\[\int \varphi \ d\Gamma(f) = \eng(\varphi f,f)-\frac12\eng(\varphi, f^2), \ \ \varphi\in\mathcal{C}.\]
The measure $\Gamma(f)$ is referred to as the \emph{energy measure} of $f$. Elements of the domain $\mathcal{F}$ represent finite energy configurations on $X$, and $\Gamma(f)$ may be regarded as the distribution of energy for the configuration $f\in\mathcal{C}$.

A nonnegative Radon measure $m$ on $X$ with $m(U)>0$ for any nonempty open $U\subset X$ is called \emph{energy dominant} for $(\mathcal{E},\mathcal{F})$ if all energy measures $\Gamma(f)$, $f\in\mathcal{C}$, are absolutely continuous with respect to $m$. Note that the original measure $\mu$ is energy dominant for $(\mathcal{E},\mathcal{F})$ if and only if $(\mathcal{E},\mathcal{F})$ admits a \emph{carr\'e du champ}, see \cite[Chapter I]{BH91}.

A sequence of functions $(f_n)_{n=1}^\infty\subset \mathcal{C}$ will be called a \emph{coordinate sequence for $(\mathcal{E},\mathcal{F})$ with respect to $m$} if the span of $\left\lbrace f_n\right\rbrace_{n=1}^\infty$ is $\mathcal{E}_1$-dense in $\mathcal{F}$ and $\Gamma(f_n)\leq m$ for any $n$. Here the notation $\Gamma(f)\leq m$ means that $\Gamma(f)$ is absolutely continuous with respect to $m$ with Radon-Nikodym derivative $d\Gamma(f)/dm$ bounded by one $m$-a.e. A coordinate sequence $(f_n)_{n=1}^\infty$ will be called \emph{point separating} if, as usual, for any distinct $x,y\in X$ there is some $f_n$ such that $f_n(x)\neq f_n(y)$. If $(f_n)_{n=1}^\infty$ is a point separating coordinate sequence, the mapping $\phi:X\to \mathbb{R}^\mathbb{N}$, given by
\[\phi(x):=(f_1(x), f_2(x),...)\]
is a bijection of $X$ onto its image $\phi(X)$. Related concepts of coordinates have already been used in \cite{HRT} and \cite{T08}. Energy dominant measures and point separating coordinate sequences can always be found. The basic idea of the following fact is standard, see for instance \cite{Hino08} or \cite{HRT}.

\begin{lemma}\label{L:coords}
Let $(\mathcal{E},\ff)$ be a regular symmetric Dirichlet form on $L_2(X,\mu)$. Then there exist a finite energy dominant measure $m_0$ and a point separating coordinate sequence $(f_n)_n$ for $(\mathcal{E},\mathcal{F})$ with respect to $m_0$.
\end{lemma}

\begin{proof} The separability of the Hilbert spaces $(\mathcal{F},\mathcal{E}_1)$ together with the regularity property (DF4) implies the existence of a countable collection of functions $\left\lbrace g_n\right\rbrace_n\subset\mathcal{C}$ such that $\lin (\left\lbrace g_n\right\rbrace_n)$ is $\mathcal{E}_1$-dense in $\mathcal{F}$. By (DF4), together with the uniform density of $C_c(X)$ in the space $C_0(X)$ of continuous functions on $X$ that vanish at infinity, any function $f\in C_0(X)$ can be uniformly approximated by a sequence of functions from $\mathcal{C}$. The Stone-Weierstrass theorem implies that $C_0(X)$ is separable, and therefore we can find a countable family $\left\lbrace h_n\right\rbrace_n\subset \mathcal{C}$ that separates the points of $X$. If a function $g_n$ has positive energy, set $\widetilde{g_n}:=\mathcal{E}(g_n)^{-1/2}\:g_n$, if it has zero energy, $\widetilde{g_n}:=g_n$. Similarly define functions $\widetilde{h_n}$. Let $(f_n)_n$ be a sequence obtained by relabelling the union $\left\lbrace \widetilde{g_n}\right\rbrace_n\cup\left\lbrace \widetilde{h_n}\right\rbrace_n$. For any summable sequence $(a_n)_n\subset (0,1)$ the sum of measures
\begin{equation}\label{E:m0}
m_0:=\sum_{n:\ \mathcal{E}(f_n)>0} a_n\:\Gamma(f_n).
\end{equation}
is a finite measure for $(\mathcal{E},\mathcal{F})$, and the energy densities satisfy $d\Gamma(f_n)/dm_0\leq 1$ $m_0$-a.e. for all $n$. It is energy dominant because $\lin (\left\lbrace f_n\right\rbrace_n)$ is dense in $\mathcal{F}$ and 
\begin{equation}\label{E:approxGamma}
|\Gamma(f)(A)^{1/2}-\Gamma(g)(A)^{1/2}|\leq \mathcal{E}(f-g)^{1/2}
\end{equation}
for any Borel set $A\subset X$ and any $f,g\in\mathcal{C}$, cf. \cite[Section 3.2]{FOT94}.
\end{proof}

\begin{remark}\mbox{}
\begin{itemize}
\item[(i)] Hino \cite{Hino08} calls an energy dominant measure for $(\mathcal{E},\mathcal{F})$ \emph{minimal} if
it is absolutely continuous with respect to any other energy dominant measure for $(\mathcal{E},\mathcal{F})$. Any two minimal energy dominant measures are mutually absolutely continuous. The measure $m_0$ as in (\ref{E:m0}) is minimal energy dominant.
\item[(ii)] Let $m$ be an energy dominant measure for $(\mathcal{E},\mathcal{F})$. It is straightforward to see that there exists a coordinate sequence for $(\mathcal{E},\mathcal{F})$ with respect to $m$ if and only if there are a countable collection of functions $\left\lbrace f_n\right\rbrace_n$ with $\lin (\left\lbrace f_n\right\rbrace_n)$ $\mathcal{E}_1$-dense in $\mathcal{F}$ and a sequence $(a_n)_n\subset (0,1)$ such that $\sum_n a_n\Gamma(f_n)\leq m$. If $m'$ is another energy dominant measure and $m\leq m'$ then trivially there is also a coordinate sequence for $m'$.
\item[(iii)] If $\mathcal{E}$ is closable with respect to $m$, we may change measure and establish $m$ as a new reference measure by time change arguments, see e.g. \cite[Section 6.2]{FOT94}. In this case $m$ may be seen as the distribution of volume.
\end{itemize}
\end{remark}

A regular symmetric Dirichlet form $(\mathcal{E},\mathcal{F})$ is called \emph{strongly local} if $\mathcal{E}(f,g)=0$ whenever $f\in\mathcal{C}$ is constant on a neighborhood of the support $\supp g$ of $g\in\mathcal{C}$, cf. \cite[Section 3.2]{FOT94}. We are interested in the question whether for a strongly local regular Dirichlet form $(\mathcal{E},\mathcal{F})$ the set $X$, together with the intrinsic metric $d_{\Gamma,m}$ induced by $(\mathcal{E},\mathcal{F})$ and an energy dominant measure $m$, forms a length space.  

Let $(\mathcal{E},\mathcal{F})$ be a strongly local Dirichlet form on $L_2(X,\mu)$. Then we can define (Radon) energy measures $\Gamma(f)$ for functions $f$ from
\begin{multline}
\mathcal{F}_\loc=\left\lbrace f\in L_{2,\loc}(X,\mu)\ | \text{ for any $K\subset X$ compact}\right.\notag\\
\left. \text{there exists some $u\in \mathcal{F}$ with $u|_K=f|_K$ $\mu$-a.e.} \right\rbrace
\end{multline}
by setting $\Gamma(f):=\Gamma(u)$, seen as a measure on $K$, if $u\in \mathcal{F}_{\loc}$ is such that $u|_K=f|_K$ $\mu$-a.e. See \cite{FOT94, LJ78, Sto10, Stu94, Stu95} for details. Now let $m$ be an energy dominant measure for $(\mathcal{E},\mathcal{F})$. For simplicity we use the symbol $\Gamma(f)$ also to denote the density  $d\Gamma(f)/dm$ of the energy measure $\Gamma(f)$ of $f\in\mathcal{F}_\loc$ with respect to this fixed measure $m$. Set 
\begin{equation}\label{E:A}
\mathcal{A}:=\left\lbrace f\in\ff_\loc\cap C(X)\ | \ \Gamma(f)\in L_\infty(X,m) \right\rbrace. 
\end{equation}
The \emph{intrinsic metric} or \emph{Carnot-Caratheodory metric} induced by $(\mathcal{E},\ff)$ and $m$ is defined by 
\begin{equation}\label{E:CCmetric}
d_{\Gamma, m}(x,y):=\sup\left\lbrace f(x)-f(y)\ |\ f\in \cal{A} \text{ with } \Gamma(f)\leq m\right\rbrace
\end{equation}
for any $x,y\in X$. If a point separating coordinate sequence for $(\mathcal{E},\mathcal{F})$ with respect to $m$ exists, then the intrinsic metric $d_{\Gamma,m}$ is a \emph{metric in the wide sense}, \cite{Sto10}, i.e. it satisfies the axioms of a metric but may attain the value $+\infty$. For investigations of $d_{\Gamma, m}$ in the context of strongly local Dirichlet forms see for instance \cite{BM95, BdMLSt, KZh, Sto10, Stu94, Stu95}. These references assume that the original reference measure $\mu$ itself is energy dominant and use it in place of $m$. However, actually one can allow arbitrary energy dominant measures $m$, and the value of $d_{\Gamma,m}$ will depend on the choice of $m$. 

A common hypothesis in the existing literature on intrinsic metrics is to require that $d_{\Gamma,m}$ induces the original topology of $X$, cf. \cite{Sto10, Stu94, Stu95}. Theorems \ref{T:general} and \ref{T:converse} below sketch a criterion for the coincidence of these topologies. They exist in various formulations and are well known. For the classical situation see for instance \cite[Theorem 11.3.3]{Dudley}. In an operator theoretic context these statements were first proved in \cite[Corollary 5.2]{Pa98} and in an abstract form for general seminorms on normed spaces in \cite[Theorems 1.8 and 1.9]{R98}. Other versions can be found in \cite{OR05, R99}. A full generalization of these statements to non-unital $C^\ast$ algebras respectively locally compact Hausdorff spaces was given in \cite[Theorem 4.1]{La07}. We restate Theorems \ref{T:general} and \ref{T:converse} for Dirichlet forms to emphasize their close connection to arguments in \cite{Sto10}.

Let $C_0(X)$ denote the space of continuous functions on $X$ that vanish at infinity and consider the space
\begin{equation}\label{E:A0}
\mathcal{A}_0:=\left\lbrace f\in\ff_\loc\cap C_0(X)\ | \ \Gamma(f)\in L_\infty(X,m) \right\rbrace. 
\end{equation}
Set $\mathcal{A}_0^1:=\left\lbrace f\in\mathcal{A}_0: \Gamma(f)\leq m\right\rbrace$. According to the proof of \cite[Lemma 5.4 (ii)]{Sto10} the set $\mathcal{A}_0^1$ is a closed subset of $C_0(X)$.

\begin{theorem}\label{T:general}
Let $(\mathcal{E},\mathcal{F})$ be a strongly local Dirichlet form on $L_2(X,\mu)$, let $m$ be an energy dominant measure for $(\mathcal{E},\mathcal{F})$ and assume there exists a point separating coordinate sequence for $(\mathcal{E},\mathcal{F})$ with respect to $m$. If  $\mathcal{A}_0^1$ is compact in $C_0(X)$, then  $d_{\Gamma, m}$ induces the original topology.
\end{theorem} 

Under the hypotheses of Theorem \ref{T:general} the cited results in \cite{La07, Pa98, R98} imply that the metric $d_{\Gamma_\mu}^0$, defined similarly as $d_{\Gamma,\mu}$ but with $\mathcal{A}_0$ in place of $\mathcal{A}$, induces the original topology on $X$.
On the other hand it follows from \cite[Appendix 4.2, Proposition 1 (a) and its proof]{Stu94} that $d_{\Gamma,m}^0$ induces the original topology on $X$ if and only if $d_{\Gamma,m}$ does. 

\begin{remark}
In the next section we will consider resistance forms, for which this coincidence of topologies can be verified directly.
\end{remark}

If $d'$ is a given metric in the wide sense on a set $X$ and $\gamma:[a,b]\to X$ is a \emph{path} in X, i.e. a continuous mapping from a closed interval $[a,b]\subset \mathbb{R}$ into $(X,d)$, then the \emph{length} of $\gamma$ is defined as
\[l(\gamma):=\sup \sum_k d'(\gamma(t_{k+1}),\gamma(t_k))\]
with the supremum taken over all finite partitions $a=t_0<t_1<...<t_N=b$ of $[a,b]$. The \emph{path metric} $d_l$ is defined by
\[d_l(x,y):=\inf\left\lbrace l(\gamma)\ |\ \text{$\gamma:[a,b]\to X$ is a path in $X$ with $x,y\in\gamma([a,b])$}\right\rbrace, \]
with $d_l(x,y):=+\infty$ if the infimum is taken over the empty set. The path metric $d_l$ always dominates the original metric, $d_l\geq d'$. The space $(X,d')$ is called a \emph{length space} if $d_l=d'$.

Now assume that $(X,d)$ is a locally compact separable metric space and $(\mathcal{E},\mathcal{F})$ is a strongly local Dirichlet form.
If the reference measure $\mu$ itself is energy dominant and the topologies induced by $d$ and $d_{\Gamma,\mu}$ coincide, then a result of Stollmann \cite[Theorem 5.2]{Sto10} implies that $(X,d_{\Gamma,\mu})$ is a length space. However, a look at the proofs of \cite[Theorems 5.1 and 5.2]{Sto10} and their background  (\cite[Lemma 5.4]{Sto10} and \cite[Lemma A.2 and Propositions A.4 and A.5]{BdMLSt}) reveals that this conclusion remains valid for any energy dominant measure $m$, provided a point separating coordinate sequence exists. A crucial ingredient seems to be a Rademacher type theorem, originally proved for general regular Dirichlet forms by Frank, Lenz and Wingert, see \cite[Theorem 4.9]{FLW} and \cite[Theorem 5.1]{Sto10}. Together with Theorem \ref{T:general} we obtain the following.

\begin{corollary}
Let $(\mathcal{E},\mathcal{F})$ be a strongly local Dirichlet form on $L_2(X,\mu)$, let $m$ be an energy dominant measure for $(\mathcal{E},\mathcal{F})$, and assume there exists a point separating coordinate sequence for $(\mathcal{E},\mathcal{F})$ with respect to $m$. If  $\mathcal{A}_0^1$ is compact in $C_0(X)$, then the metric space $(X,d_{\Gamma, m})$ is a length space.
\end{corollary}

\begin{remark} Assume $\mathcal{A}_0^1$ is compact in $C_0(X)$.
\begin{enumerate}
\item[(i)] Then for any two distinct points $x,y\in X$ with $d_{\Gamma,m}^0(x,y)<+\infty$ and any sequence $(f_n)_n\subset \mathcal{A}^1_0$ such that $d_{\Gamma,m}^0(x,y)=\lim_n (f_n(x)-f_n(y))$ there is a uniformly convergent subsequence $(f_{n_k})_k$ with limit $g\in C_0(X)$, hence $d_{\Gamma,m}^0(x,y)=g(x)-g(y)$.
\item[(ii)] Since the topologies induced by $d_{\Gamma,m}^0$ and $d$ and $d_{\Gamma,m}$ coincide by Theorem \ref{T:general} and \cite[Appendix 4.2, Proposition 1 (a)]{Stu94}, the distance function $d_x$ , given by
\[d_x(z):=d_{\Gamma,m}(x,z),\ \ z\in X,\]
is a maximizing element in (\ref{E:CCmetric}), i.e. $d_x(y)=d_{\Gamma,m}(x,y)$. This was shown by Sturm, see \cite[Lemma 1']{Stu94} and \cite[Lemma 1]{Stu95}. In general $d_x$ will not be an element of $\mathcal{A}_0$.
\end{enumerate}
\end{remark} 

For compact spaces a converse of Theorem \ref{T:general} is a simple consequence of the Arzel\`{a}-Ascoli theorem.
\begin{theorem}\label{T:converse}
Assume $(X,d)$ is compact. Let $(\mathcal{E},\mathcal{F})$ be a strongly local Dirichlet form on $L_2(X,\mu)$ and $m$ an energy dominant measure for $(\mathcal{E},\mathcal{F})$. If the topologies induced by $d$ and $d_{\Gamma,m}$ coincide, then $\mathcal{A}^1_0$ is compact in $C(X)$.
\end{theorem}

\section{Length spaces induced by resistance forms}\label{S:resistance}

In this section we investigate resistance forms in the sense of Kigami and show that in the compact case they always produce a geodesic space. 
We recall the definition, see \cite[Definition 2.8]{Kig03} or \cite[Definition 3.1]{Kig12}. Given a set $X$, a pair $(\mathcal{E},\mathcal{F})$ is called a \emph{resistance form} on $X$ if
\begin{enumerate}
\item[(RF1)] $\mathcal{E}:\mathcal{F}\times\mathcal{F}\to\mathbb{R}$ is a nonnegative definite symmetric bilinear form on a vector space $\mathcal{F}$ of real valued functions on $X$, and $\mathcal{E}(u)=0$ if and only if $u$ is constant on $X$,
\item[(RF2)] $(\mathcal{F}/\sim,\mathcal{E})$ is a Hilbert space; here $\sim$ is the equivalence relation on $\mathcal{F}$ given by $u\sim v$ if and only if $u-v$ is constant on $X$,
\item[(RF3)] $\mathcal{F}$ separates the points of $X$,
\item[(RF4)] For any $x,y\in X$ the expression
\[d_R(x,y) := \sup\left\lbrace |u(x)-u(y)|^2~:~u\in\ff,\ \eng(u)\leq 1\right\rbrace\]
is finite,
\item[(RF5)] $(\mathcal{E},\mathcal{F})$ has the \emph{Markov property}, i.e. $u\in\mathcal{F}$ implies $(0\vee u)\wedge 1\in\mathcal{F}$ and 
\[\mathcal{E}((0\vee u)\wedge 1)\leq \mathcal{E}(u).\]
\end{enumerate}
Comprehensive background can be found in \cite{Kig93, Kig01, Kig03}. For resistance forms the length space property can be verified independently of Theorem \ref{T:general}. 

Take $X$ to be a nonempty set and let $(\mathcal{E},\mathcal{F})$ be a resistance form on $X$. Then $d_R$ as defined in (RF4) is a metric on $X$, the so-called \emph{resistance metric} associated with $(\mathcal{E},\mathcal{F})$, cf. \cite[Definition 2.3.2]{Kig01} and \cite[Definition 2.11]{Kig03}. There is a one-to-one correspondence between resistance forms and resistance metrics, see \cite[Theorems 2.3.4 and 2.3.6]{Kig01}.

The inequality 
\begin{equation}\label{E:resistanceest}
|u(x)-u(y)|^2\leq d_R(x,y)\mathcal{E}(u)
\end{equation}
holds for any $x,y\in X$ and $u\in \ff$, showing that the space $\ff$ is a subspace of the space of $1/2$-H\"older continuous functions on $(X,d_R)$, cf. \cite[Section 2]{Kig03}. In particular, $\mathcal{F}\subset C(X)$. If the space $(X,d_R)$ is compact,
then the resistance form $(\mathcal{E},\mathcal{F})$ is seen to be \emph{regular}, i.e. $\mathcal{F}$ is uniformly dense in $C(X)$, \cite{Kig12}. Further, it is known, \cite[Section 2.3]{Kig01} that
\[
d_R^{1/2}(x,y) := \sup\left\lbrace |f(x)-f(y)|~:~\eng(f)\leq 1\right\rbrace
\]
is a metric which induces the same topology as $d_R$.

If the space $(X,d_R)$ is endowed with a suitable measure, a given resistance form $(\mathcal{E},\mathcal{F})$ induces a Dirichlet form in the sense of \cite{FOT94} as discussed in the preceding section. Let $\mu$ be a finite nonnegative Borel regular measure on the space $(X,d_R)$ with $\mu(U)>0$ for all nonempty open $U\subset X$. Then $\mathcal{F}\subset L_2(X,\mu)$, and $(\mathcal{E},\mathcal{F})$ is a regular symmetric Dirichlet form on $L_2(X,\mu)$. For a fixed measure $\mu$ we may therefore consider energy measures as in the previous section, and Lemma \ref{L:coords} remains valid. 

\begin{lemma}\label{L:separatepoints}
Let $(\mathcal{E},\ff)$ be a resistance form on $X$ such that $(X,d_R)$ is compact and let $m$ be a finite energy dominant measure on $(X,d_R)$. Then any coordinate sequence $(f_n)_{n=1}^\infty$ for $(\mathcal{E},\mathcal{F})$ with respect to $m$ is point separating.
\end{lemma}
\begin{proof}
Assume that $x,y\in X$ are two distinct points with $f_n(x)=f_n(y)$ for all $n$. Then $f(x)=f(y)$ for all $f\in\mathcal{F}$ by linearity, approximation and (\ref{E:resistanceest}), contradicting $d_R(x,y)>0$.
\end{proof}

We introduce yet another metric $d_\phi$, now in terms of coordinates. For a fixed point separating coordinate sequence $(f_n)_{n=1}^\infty$ set
\[d_\phi(x,y):=\sup\left\lbrace|f_k(x)-f_k(y)|~:~k=1,2,\ldots\right\rbrace\ ,\ \ x,y\in X.\]
A resistance form $(\mathcal{E},\mathcal{F})$ is called \emph{local} if it is local (and therefore strongly local) in the Dirichlet form sense. Let $(\mathcal{E},\mathcal{F})$ be a local resistance form. If the topologies induced by $d_R$ and $d_{\Gamma,m}$ coincide we may again use \cite[Theorem 5.2]{Sto10} to conclude $(X,d_{\Gamma,m})$ is a length space. In order to ensure the required coincidence of topologies we will now compare the metrics $d_R^{1/2}$, $d_{\Gamma,m}$ and $d_\phi$. 

\begin{theorem}\label{T:geodesicresistance}
Suppose $(\mathcal{E},\ff)$ is a local resistance form on $X$ such that $(X,d_R)$ is compact, let $m$ be an energy dominant measure for $(\mathcal{E},\mathcal{F})$, and assume there exists a coordinate sequence for $(\mathcal{E},\mathcal{F})$ with respect to $m$. Then the topologies induced by $d_R$ and $d_{\Gamma,m}$ coincide, and $(X,d_{\Gamma,m})$ is a length space.
\end{theorem}

\begin{proof}
Let $(f_n)_{n=1}^\infty$ be a coordinate sequence for $(\mathcal{E},\ff)$. Because 
\[(f_n)_{n=1}^\infty\subset \left\lbrace f\in \mathcal{A}: \Gamma(f)\leq m\right\rbrace\subset \left\lbrace f \in\mathcal{F}\ | \ \mathcal{E}(f)\leq 1\right\rbrace,\]
the suprema over these sets increase, thus
\[d_\phi(x,y) \leq d_{\Gamma,m}(x,y) \leq d_R^{1/2}(x,y)\]
for all $x,y\in X$, and the open $d_\phi$-ball centered at $x_0$ with radius $r>0$ contains the open $d_R^{1/2}$-ball centered at $x_0$ with radius $r$, hence also an open $d_R$-ball centered at $x_0$. For arbitrary fixed $x_0$ on the other hand there must be some $\varepsilon>0$ such that an open $d_\phi$-ball centered at $x_0$ with radius $\varepsilon$ is contained in the open $d_R$-ball of radius one centered at $x_0$. If not, then we could find a sequence $(x_k)_{k=1}^\infty$ that does not converge in the $d_R$-ball with respect to $d_{R}$ but converges to $x_0$ with respect to $d_\phi$. Then by the definition of $d_\phi$ all the differences $|f_n(x_k)-f_n(x_0)|$ would go to zero. By compactness there exist a subsequence $(x_{k_l})_{l=1}^\infty$ of $(x_k)_{k=1}^\infty$ and some $x_\infty\neq x_0$ to which $(x_{k_l})_{l=1}^\infty$ converges with respect to $d_R$. The continuity of the coordinates $f_n$ with respect to $d_R$ implies that also $|f_n(x_{k_l})-f_n(x_\infty)|$ goes to zero for all $n$, what is a contradiction, because $(f_n)_n$ separates points.
\end{proof}

A metric space $(X,d)$ is called \emph{geodesic} if for any two distinct points $x,y\in X$ there exists a path $\gamma$ of length $d(x,y)$. If a metric space is a length space and complete, then it is a geodesic space, as was shown in \cite[Theorem 1]{Stu95}. 

\begin{corollary}
Suppose $(\mathcal{E},\ff)$ is a local resistance form on $X$ such that $(X,d_R)$ is compact, let $m$ be an energy dominant measure for $(\mathcal{E},\mathcal{F})$, and assume there exists a coordinate sequence for $(\mathcal{E},\mathcal{F})$ with respect to $m$. Then the space $(X,d_{\Gamma,m})$ is compact and therefore complete and geodesic.
\end{corollary}
\begin{proof}
The Corollary follows because $(X,d_R)$-compactness implies $(X,d_R^{1/2})$-compactness and $d_R^{1/2}\geq d_{\Gamma,m}$.
\end{proof}

\begin{remark}
The converse conclusion is not valid: There are spaces carrying a regular resistance form that are not $d_R$-compact but can be equipped with a measure $m$ such that they become $d_{\Gamma,m}$-compact. See Example \ref{Ex:notvalid} below.
\end{remark}

A special situation arises if $(\mathcal{E},\mathcal{F})$ is a resistance form on a dendrite, \cite{Kig95}. A \emph{dendrite} (or \emph{tree}) is an arcwise connected topological space that has no subset homeomorphic to a circle, cf. \cite[Definition 0.6]{Kig95}. Given two points $x,y$ in a dendrite $X$ there exists a unique (up to reparametrization) path $\gamma_{x,y}:[0,1]\to X$ such that $\gamma_{x,y}(0)=x$ and $\gamma_{x,y}(1)=y$. A metric $d$ on a dendrite $X$ is called a \emph{shortest path metric} if for any $x,y\in X$ and any $z\in \gamma_{x,y}([0,1])$ we have 
\[d(x,y)=d(x,z)+d(z,y).\]

\begin{theorem}\label{T:dendrite}
Suppose that $X$ is a dendrite, $(\mathcal{E},\ff)$ is a local resistance form on $X$ such that $(X,d_R)$ is compact, $m$ is a finite energy dominant measure for $(\mathcal{E},\mathcal{F})$, and assume there exists a coordinate sequence for $(\mathcal{E},\mathcal{F})$ with respect to $m$. Then $d_{\Gamma,m}$ itself is a shortest path metric. 
\end{theorem} 

\begin{proof}
By Theorem \ref{T:geodesicresistance} the space $(X,d_{\Gamma,m})$ is a length space. Hence if $x,y\in X$ are two distinct points and $\gamma_{x,y}$ is the unique path such that $\gamma_{x,y}(0)=x$ and $\gamma_{x,y}(1)=y$, we have $d_{\Gamma,m}(x,y)=l(\gamma_{x,y})$. If $z$ is yet another point on $\gamma_{x,y}([0,1])$ then we have
\[\gamma_{x,y}(t)=
\begin{cases} \gamma_{x,z}(2t)\ \ 0\leq t\leq \frac{1}{2}\\
\gamma_{z,y}(2t)\ \ \frac{1}{2}\leq t\leq 1,
\end{cases}
\]
where $\gamma_{x,z}$ and $\gamma_{z,y}$ are the uniquely determined paths joining $x$ and $z$, respectively $z$ and $y$, and the additivity of the path length yields
\[d_{\Gamma,m}(x,y)=l(\gamma_{x,y})=l(\gamma_{x,z})+l(\gamma_{z,y})=d_{\Gamma,m}(x,z)+d_{\Gamma,m}(z,y).\]
\end{proof}

\begin{remark}
Kigami has shown in \cite[Proposition 5.1]{Kig95} that any shortest path metric $d$ on a dendrite $X$ is a resistance metric on $X$. Under additional assumptions this implies that $d$ is a resistance metric  associated to a regular resistance form, \cite[Theorem 5.4]{Kig95}.
In the situation of Theorem \ref{T:dendrite} it follows that $d_{\Gamma,m}$ is a resistance metric in the sense of \cite[Definition 2.3.2]{Kig01}.
\end{remark}

\begin{examples}\label{Ex:notvalid}
Consider a dendrite $X$ consisting of countably many copies $[p,q_n]$, $n\in\mathbb{N}$, of the unit interval $[0,1]$, glued together at the left interval end point $p$. Set 
\[\mathcal{F}:=\left\lbrace f:X\to\mathbb{R}: \text{ $f$ is absolutely continuous on each $[p,q_n]$ and} \sum_n\int_p^{q_n} |f'(x)|^2dx<\infty\right\rbrace\]
and 
\[\mathcal{E}(f):=\sum_n\int_p^{q_n} |f'(x)|^2dx, \ \ f\in\mathcal{F}.\]
Then $\mathcal{E}$ is a resistance form on $X$. Obviously $d_R(p,q_n)=1$, and for distinct $m,n\in\mathbb{N}$ we have
\[d_R(q_m,q_n)=2.\]
Hence the sequence $(q_n)_n$ is $d_R$-bounded but has no $d_R$-convergent subsequence, so $X$ is not $d_R$-compact. On the other hand we can equip $X$ with a suitable measure $m$ such that it becomes $d_{\Gamma,m}$-compact. Let $(a_n)_n$ be a bounded sequence of positive real numbers converging to zero and set
\[m:=\sum_n a_n dx|_{[p,q_n]}.\]
Then $m$ is a finite energy dominant measure for $(\mathcal{E},\mathcal{F})$ and on $[p,q_m]$ the density of $\Gamma(f)$ with respect to $m$ 
is given by $a_m^{-1}|f'|^2$. Now let $(p_n)_n$ be a $d_{\Gamma,m}$-bounded sequence. If it has a subsequence that is entirely contained in one segment $[p, q_m]$ then it has a subsequence $(p_{n_k})_k$ that converges in Euclidean metric to some point $q\in[p,q_m]$. For any function $f\in\mathcal{F}$ with $\Gamma(f)\leq m$ we have $|f'| \leq \sqrt{a_m}$ a.e. on $[p,q_m]$ and therefore 
$|f(q)-f(p_{n_k})|\leq \sqrt{a_m}|q-p_{n_k}|$. Hence $(p_{n_k})_k$ $d_{\Gamma,m}$-converges to $q$. If no subsequence of $(p_n)_n$ is contained in a single segment then there must be a subsequence $(p_{n_k})_k$ such that $p_{n_k}\in [p, q_k]$ for any $k$. For any $k$ there is some $f$ with $|f_k'| \leq \sqrt{a_k}$ a.e. on $[p,q_k]$ such that 
\[d_{\Gamma,m}(p,p_{n_k})\leq |f_k(p)-f_k(p_{n_k})|+\frac1k\leq \sqrt{a_k}+\frac1k\ ,\]
hence $(p_{n_k})_k$ is $d_{\Gamma,m}$-convergent to $p$.
\end{examples}

\section[Dirac operators]{Dirac operators}\label{S:Dirac}

In this section we introduce Dirac operators and spectral triples related to Dirichlet forms. Our considerations are based on the first order theory proposed by Cipriani and Sauvageot in \cite{CS03, CS09, Sau89, Sau90} and developed in \cite{CGIS12, HRT, IRT}. Related constructions can be found in \cite{Eb99, W99, W00}.

As in Section \ref{S:Dirichlet} let $X$ be a locally compact separable metric space and $\mu$ be a nonnegative Radon measure on $X$ with $\mu(U)>0$ for any nonempty open $U\subset X$. Let $(\mathcal{E},\mathcal{F})$ be a regular symmetric Dirichlet form on $L_2(X,\mu)$. According to the Beurling-Deny decomposition the form $\mathcal{E}$ uniquely decomposes
into a strongly local, a pure jump and a killing part, see \cite[Theorem 3.2.1]{FOT94}. In this section we assume that the killing part of $\mathcal{E}$ is zero. Recall that we write $\mathcal{C}:=C_c(X)\cap \mathcal{F}$ and that the mutual energy measure of two functions $f,g\in\mathcal{C}$ is denoted by $\Gamma(f,g)$.  We equip the space $\mathcal{C}\otimes\mathcal{C}$ with a bilinear form, determined by
\[\left\langle a\otimes b,c\otimes d\right\rangle_\mathcal{H}=\int_Xbd\:d\Gamma(a,c).\]
The right hand side is the integral of the product $bd\in\mathcal{C}$ with respect to the mutual energy measure $\Gamma(a,c)$ of $a$ and $c$. This bilinear form is nonnegative definite, hence it defines a seminorm on $\mathcal{C}\otimes\mathcal{C}$. Let $\mathcal{H}$ denote the Hilbert space obtained by first factoring out zero seminorm elements and then completing. Following \cite{CS03} we
refer to it as the \emph{space of differential $1$-forms} associated with $(\mathcal{E},\mathcal{F})$. 

\begin{examples}\label{Ex:classical}
If $X=M$ is a smooth compact Riemannian manifold without boundary and $dvol$ the Riemannian volume on $M$, then the closure in $L_2(M, dvol)$ of 
\[\mathcal{E}(f)=\int_M \left\|df\right\|^2_{T^\ast M}\:dvol,\ \ f\in C^\infty(M),\]
where $df$ denotes the exterior derivative of $f$, is a strongly local Dirichlet form $(\mathcal{E},\mathcal{F})$ on $L_2(M, dvol)$. For a simple tensor $f\otimes g\in\mathcal{C}\otimes\mathcal{C}$ we observe that 
\[\left\|f\otimes g\right\|_\mathcal{H}^2=\int_M\left\|gdf\right\|_{T^\ast M}^2dvol.\]
In this case $\mathcal{H}$ is isometrically isomorphic to the space $L_2(M, T^\ast M, dvol)$ of $L_2$-differential $1$-forms.
\end{examples}

The space $\mathcal{H}$ can be made into a $\mathcal{C}$-$\mathcal{C}$-bimodule:  Setting 
\begin{equation}\label{E:actionsofC}
c(a\otimes b):=(ac)\otimes b-c\otimes (ab)\ \ \text{ and }\ \ (a\otimes b)c:=a\otimes (bc)
\end{equation}
for $a,b,c\in\mathcal{C}$ and extending linearly we observe the bounds
\[\left\|\sum_{i=1}^n c(a_i\otimes b_i)\right\|_\mathcal{H}\leq \sup_X |c| \left\|\sum_{i=1}^n a_i\otimes b_i\right\|_\mathcal{H}\]
and
\[\left\|\sum_{i=1}^n (a_i\otimes b_i)c\right\|_\mathcal{H}\leq \sup_X|c| \left\|\sum_{i=1}^n a_i\otimes b_i\right\|_\mathcal{H},\] 
and by continuity we can extend further to obtain uniformly bounded left and right actions of $\mathcal{C}$ on $\mathcal{H}$. See \cite{CS03, IRT}. 

The definition $$\partial a:=a\otimes\mathbf{1}$$ yields a derivation operator $\partial:\mathcal{C}\to\mathcal{H}$ such that 
\[\left\|\partial a\right\|^2_\mathcal{H}= \mathcal{E}(a)\] 
and the Leibniz rule holds,
\begin{equation}\label{E:Leibniz}
\partial(ab)=a\partial b+(\partial a)b, \ \ a,b\in\mathcal{C}.
\end{equation}

\begin{remark}
By approximation, the right action is also well defined for elements $c$ of the space $\mathcal{B}_b(X)$ of bounded Borel functions on $X$, and the space $\mathcal{H}$ agrees with the Hilbert space obtained by factoring $\mathcal{B}_b(X)\otimes \mathcal{C}$ and completing similarly as before.
\end{remark}

\begin{remark}
We give a short comment concerning the above construction in the case of purely non local Dirichlet forms. For simplicity assume that $X$ is compact such that $\mathbf{1}\in\mathcal{C}$.
A customary algebraic standard definition is to consider the tensor product $\mathcal{C}\otimes\mathcal{C}$, endowed with the $\mathcal{C}$-actions $c(a\otimes b):=(ca)\otimes b$ and $(a\otimes b)c:=a\otimes (bc)$, and a derivation $d:\mathcal{C}\to\mathcal{C}\otimes\mathcal{C}$, given by $da:=a\otimes \mathbf{1}-\mathbf{1}\otimes a$ up to a sign convention, see e.g. \cite{GVF}. As $\mathcal{C}$ is an algebra of functions, we have $(a\otimes b)(x,y)=a(x)b(y)$ for any $x,y,\in X$, and in particular
\[(cda)(x,y)=c(x)(a(x)-a(y)) \ \text{ and }\ \ ((da)c)(x,y)=(a(x)-a(y))c(y).\]
The difference of $(da)c-(\partial a)c$ has zero  
seminorm, $\left\|\mathbf{1}\otimes (ac)\right\|_\mathcal{H}^2=0$. The definition of the left action in (\ref{E:actionsofC}) produces the Leibniz rule (\ref{E:Leibniz}), note also that $c\partial a$, defined according to (\ref{E:actionsofC}), agrees in $\mathcal{C}\otimes\mathcal{C}$ with $cda$, defined using the left action in the present remark.
If for instance $(\mathcal{E},\mathcal{F})$ is a purely nonlocal Dirichlet form with jump measure $J$, cf. \cite[Theorem 3.2.1]{FOT94},
\[\mathcal{E}(f)=\frac12\int_X\int_X (f(x)-f(y))^ 2\:J(dxdy),\ \ f\in\mathcal{F},\]
then 
\[\left\|g\partial f\right\|_\mathcal{H}^ 2=\frac12\int_X g(x)^ 2\int_X (f(x)-f(y))^ 2\:J(dxdy),\]
and the difference $df$, given by $df(x,y)=f(x)-f(y)$,  is a representative of the $\mathcal{H}$-equivalence class $\partial f$. If moreover the jump measure $J$ is concentrated on 
\[\left\lbrace (x,y)\in X\times X: 0<d(x,y)<\varepsilon\right\rbrace,\]
$f$ is supported in a bounded set $A\subset X$ and $g$ is supported outside $\left\lbrace x\in X: \dist(x,A)<\varepsilon\right\rbrace$, then $g\partial f$ is zero in $\mathcal{H}$.
\end{remark}

Let the space $\mathcal{C}$ be equipped with the norm $\left\|f\right\|_{\mathcal{C}}:=\mathcal{E}_1(f)^{1/2}+\sup_{x\in X}|f(x)|$ and let $\mathcal{C}^\ast$ denote the dual space of $\mathcal{C}$, equipped with the usual norm. Note that $\mathcal{C}\subset L_2(X,\mu)\subset \mathcal{C}^\ast$. We write $\left\langle u,\varphi\right\rangle=u(\varphi)$ to denote the dual pairing of $u\in\mathcal{C}^\ast$ and $\varphi\in\mathcal{C}$. For $\omega\in\mathcal{H}$ let $\partial^\ast\omega$ be the element of $\mathcal{C}^\ast$ defined by 
\[(\partial^\ast\omega)(\varphi):=\left\langle \omega,\partial\varphi\right\rangle_\mathcal{H}, \ \ \varphi\in\mathcal{C}.\]
It is straightforward to see that $\partial^\ast$ is a bounded linear operator $\partial^\ast:\mathcal{H}\to\mathcal{C}^\ast$. The operator $\partial$ extends to a densely defined closed linear operator $\partial: L_2(X,\mu)\to\mathcal{H}$ with domain $dom\:\partial=\mathcal{F}$. The restriction of $\partial^\ast$ to 
\begin{multline}
dom\:\partial^ \ast=\Big\lbrace \omega\in\mathcal{H}: \text{ there exists }u^\ast\in L_2(X,\mu)\text{ such that}\notag\\
\text{ $\left\langle u^\ast,\varphi\right\rangle_{L_2(X,\mu)}=\left\langle \omega,\partial \varphi\right\rangle_\mathcal{H}$ for all $\varphi\in\mathcal{F}$}\Big\rbrace
\end{multline}
is the adjoint of $\partial$, i.e. the unbounded linear operator $\partial^\ast:\mathcal{H}\to L_2(X,\mu)$ such that for all $\omega\in dom\:\partial^\ast$ we have
\begin{equation}\label{E:IbP}
\left\langle \partial^\ast \omega, \varphi\right\rangle_{L_2(X,\mu)}=\left\langle \omega, \partial \varphi\right\rangle_\mathcal{H}, \ \ \varphi\in\mathcal{F}.
\end{equation}
By general theory $(\partial^\ast, dom\:\partial^\ast)$ is closed and densely defined. 

\begin{examples}
In the situation of Examples \ref{Ex:classical} the operator $\partial$ coincides with the exterior derivative $d$, seen as an unbounded closed linear operator from $L_2(M, dvol)$ into $L_2(M, T^\ast M, dvol)$.
\end{examples} 

Let $(L, dom\:L)$ denote the infinitesimal $L_2(X,\mu)$-generator of $(\mathcal{E},\mathcal{F})$, i.e. the nonpositive definite self-adjoint operator $L$ on $L_2(X,\mu)$ with domain $dom\:L\subset \mathcal{F}$ such that 
\[\mathcal{E}(f,g)=-\left\langle f, Lg\right\rangle_{L_2(X,\mu)}\]
for all $f\in\mathcal{F}$ and $g\in dom\:L$. Note that $\partial^\ast\partial g=-L g$, $g\in dom\:L$, as was already proved in \cite{CS09}. The image $Im\:\partial$ of $\partial$ is a closed subspace of $\mathcal{H}$: We have
\[ker\:L=\left\lbrace f\in L_2(X,\mu): \mathcal{E}(f)=0\right\rbrace,\]
and $\mathcal{F}$ decomposes orthogonally into $ker\:L$ and its complement in $\mathcal{F}$,
\[\mathcal{F}= ker\:L\oplus (\ker\:L)^\bot_\mathcal{F}.\]
The space $((\ker\:L)^\bot_\mathcal{F},\mathcal{E})$ is Hilbert, and therefore the the image of $(\ker\:L)^\bot_\mathcal{F}$ under $\partial$ is a closed subspace of $\mathcal{H}$. However, as $\partial f=0$ for all $f\in ker\:L$, this image is just $Im\:\partial$.
Consequently $\mathcal{H}$ decomposes orthogonally into $Im\:\partial$ and $ker\:\partial^\ast$,
\[\mathcal{H}=Im\:\partial \oplus ker\:\partial^\ast,\]
and we have $dom\:\partial^\ast=\left\lbrace \partial f: f\in dom\:L\right\rbrace \oplus ker\:\partial^\ast$.

From now on we consider the natural complexifications of $L_2(X,\mu)$, $\mathcal{E}$, $\mathcal{F}$, $\Gamma$, $\mathcal{H}$, $\mathcal{C}$ and the operators $\partial$ and $\partial^\ast$, and for simplicity we denote them by the same symbols. The algebra $\mathcal{C}$ becomes involutive by complex conjugation. 

The Hilbert space
\[\mathbb{H}:=L_2(X,\mu)\oplus \mathcal{H}\]
carries the natural scalar product 
\[\left\langle (f, \omega), (g, \eta)\right\rangle_{\mathbb{H}}:=\left\langle f,g\right\rangle_{L_2(X,\mu)}+\left\langle \omega,\eta\right\rangle_{\mathcal{H}}.\] 
Put $dom\:\mathbb{D}:=\mathcal{F}\oplus dom\:\partial^\ast$ and define an unbounded linear operator $\mathbb{D}:\mathbb{H}\to\mathbb{H}$ by 
\begin{equation}\label{E:Diracdef}
\mathbb{D}(f,\omega):=(\partial^\ast\omega, \partial f), \ \ (f,\omega)\in dom\:\mathbb{D}.
\end{equation}
To $\mathbb{D}$ we refer as the \emph{Dirac operator} associated with $(\mathcal{E},\mathcal{F})$. In matrix notation its definition reads
\[\mathbb{D}=\left(\begin{array}{rr}
0 &\partial^\ast\\
\partial &0
\end{array}\right).\]
This definition of a Dirac operator follows sign and complexity conventions often used in geometry and differs slightly from the definition in \cite{HTb}. 

\begin{lemma}
The operator $(\mathbb{D}, dom\:\mathbb{D})$ is self-adjoint on $\mathbb{H}$.
\end{lemma}

This lemma is not difficult to see: A direct calculation shows that $\mathbb{D}$ is symmetric, see \cite[Theorem 3.1]{HTb}, hence also $\overline{\mathbb{D}}$ is symmetric. By the closedness of $\partial$ and $\partial^\ast$ we have $\overline{\mathbb{D}}=\mathbb{D}$ and due to the matrix structure of $\mathbb{D}$ also $\mathbb{D}^\ast=\overline{\mathbb{D}}$. The symmetry of $\mathbb{D}^\ast$ then implies that $\overline{\mathbb{D}}$ is self-adjoint, see for instance \cite[Theorem 5.20]{Wei80}.

We are particularly interested in the special case where the generator $L$ of $(\mathcal{E},\mathcal{F})$ has pure point spectrum , i.e. there are an increasing sequence $0<\lambda_1\leq \lambda_2\leq ...$ of nonzero eigenvalues $\lambda_i$ of $-L$, with possibly infinite multiplicities taken into account, and an orthonormal basis $\left\lbrace \varphi_j\right\rbrace_{j=1}^\infty$ in $L_2(X,m)$ of corresponding eigenfunctions such that 
\begin{equation}\label{E:specL}
-Lf=\sum_{j=1}^\infty \lambda_j\left\langle f,\varphi_j\right\rangle_{L_2(X,\mu)}\varphi_j, \ \ f\in dom\:L,
\end{equation}
and zero itself may be an eigenvalue of infinite multiplicity. See \cite{RSI}. In this case the Dirac operator $\mathbb{D}$ rewrites as follows.

\begin{lemma}\label{L:Diracrep}
If the generator $L$ of $(\mathcal{E},\mathcal{F})$ has pure point spectrum with spectral representation (\ref{E:specL}) then $\mathbb{D}$ admits the spectral representation
\begin{equation}\label{E:Diracrep}
\mathbb{D}v=\sum_{j=1}^\infty \lambda_j^{1/2}\left\langle v,v_j\right\rangle_\mathbb{H} v_j - \sum_{j=1}^\infty \lambda_j^{1/2} \left\langle v,w_j\right\rangle_\mathbb{H} w_j, \ \ v\in dom\:\mathbb{D},
\end{equation}
where 
\[v_j=\frac{1}{\sqrt{2}}(\varphi_j, \lambda_j^{-1/2} \partial\varphi_j) \ \text{ and }\ \ w_j=\frac{1}{\sqrt{2}}(\varphi_j, -\lambda_j^{-1/2}\partial\varphi_j), \ \ j=1,2,...\]
In general, zero may be an eigenvalue of $\mathbb{D}$ of infinite multiplicity.
\end{lemma}

To prove the lemma we investigate the square $\mathbb{D}^2$ of $\mathbb{D}$. Set 
\[dom\:\Delta_{1}:=\left\lbrace \omega\in dom\:\partial^\ast: \partial^\ast\omega\in\mathcal{F}\right\rbrace\] 
and 
\[\Delta_{1}\omega:=\partial\partial^\ast \omega, \ \ \omega \in dom\:\Delta_{1}.\] 
The restriction $-L_\bot$ of $-L$ to $(ker\:L)^\bot_{L_2(X,\mu)}$ has a nonnegative and bounded inverse $(-L_\bot)^{-1}$, and by $(-L_\bot)^{-1/2}$ we denote its square root. Set
\[Uf:=\partial ((-L_\bot)^{-1/2}f),\ \ f\in L_2(X,\mu).\]
Then $U$ is a unitary transformation from $(ker\:L)^\bot_{L_2(X,\mu)}$ onto a subspace of $\mathcal{H}$, and it is not difficult to see that
\[U((ker\:L)^\bot_{dom\:L})= Im\:\partial \cap dom\:\Delta_1=\left\lbrace \partial g: Lg\in\mathcal{F}\right\rbrace.\] 
For $f\in (ker\:L)^\bot_{dom\:L}$ we have $\Delta_1 Uf=\partial ((-L)^{-1/2}f)=U L f$. Moreover, the $1$-forms
\[\omega_j:=U\varphi_i=\lambda_j^{-1/2}\partial \varphi_j, \ \ j=1,2,...\]
yield an orthonormal basis of $Im\:\partial$, and $\Delta_1 \omega_j=\lambda_j\omega_j$. By choosing a suitable orthonormal basis of $ker\:\partial^\ast$ (note that $\mathcal{H}$ is separable), we can obtain an orthonormal basis of $\mathcal{H}$ such that for any $\omega\in dom\:\Delta_1$ we have
\[\Delta_1\omega=\sum_{j=1}^\infty \lambda_j \left\langle \omega, \omega_j\right\rangle_{\mathcal{H}}\omega_j.\]
We also observe that
\[dom\:\Delta_1=\left\lbrace \omega\in\mathcal{H}: \sum_{j=1}^\infty \lambda_j |\left\langle \omega, \omega_i\right\rangle_\mathcal{H}|^2<+\infty \right\rbrace.\]
Therefore the operator $(\Delta_1, dom\:\Delta_1)$ is self-adjoint on $\mathcal{H}$ with eigenvalues $\lambda_1, \lambda_2, ...$, and possibly also zero is an eigenvalue.

\begin{remark}
It follows from the results in \cite{CS09, IRT} that for resistance forms on finitely ramified fractals (such as for instance the Sierpinski gasket) the space $ker\:\partial^\ast$ is infinite dimensional (even if zero is not an eigenvalue of $L$) and therefore zero is an eigenvalue of $\Delta_1$ of infinite multiplicity. A precise statement for the Sierpinski gasket is \cite[Theorem 3.9]{CGIS13}. See also \cite{HTa} for more general metric spaces.
\end{remark}

The square $\mathbb{D}^2$ of $\mathbb{D}$ is given by
\[\mathbb{D}^2=\left(\begin{array}{rr} 
\partial^\ast\partial & 0\\
0 &\partial\partial^\ast\end{array}\right),\]
and it is straightforward to see that its domain $dom\:\mathbb{D}^2:=\left\lbrace v\in\mathbb{H}: \mathbb{D}v\in dom\:\mathbb{D}\right\rbrace$ coincides with $dom\:L\oplus dom\:\Delta_1$.
By the spectral theorem also $(\mathbb{D}^2, dom\:\mathbb{D}^2)$ is self-adjoint on $\mathbb{H}$. Put
\[v_{j,0}:=(\varphi_j, 0) \text{ and }\ v_{j,1}:=(0,\omega_j),\ \ i=1,2,...\]

\begin{lemma}
Assume that $L$ has pure point spectrum with spectral representation (\ref{E:specL}). Then the operator $(\mathbb{D}^2, dom\:\mathbb{D}^2)$ admits the spectral representation 
\[\mathbb{D}^2 v= \sum_{i=0,1}\sum_{j=1}^\infty\lambda_j \left\langle v,v_{j,i}\right\rangle_{\mathbb{H}} v_{j,i}, \ \ v\in dom\:\mathbb{D}^2.\]
In general, zero may be an eigenvalue of $\mathbb{D}^2$ of infinite multiplicity.
\end{lemma}

Now Lemma \ref{L:Diracrep} is proved quickly.
\begin{proof} 
If
\[\mathbb{D}=\int_\mathbb{R} xdE_x\]
is the spectral representation of $\mathbb{D}$ then we have
\[\int_\mathbb{R} x^2\:d\left\langle E_xv,v\right\rangle_\mathbb{H}=\left\langle v,\mathbb{D}^2v\right\rangle_{\mathbb{H}}=\sum_{j=1}^\infty \lambda_j|\left\langle v,v_j\right\rangle_\mathbb{H}|^2\]
for any $v\in dom\:\mathbb{D}^2$. By functional calculus it follows that the measures $d\left\langle E_xv,v\right\rangle_\mathbb{H}$ are supported on the discrete set 
\[\left\lbrace -\lambda_j^{1/2}\right\rbrace_{j=1}^\infty\cup\left\lbrace 0\right\rbrace\cup \left\lbrace \lambda_j^{1/2}\right\rbrace_{j=1}^\infty,\] 
and a direct calculation shows that  $v_j$ and $w_j$ are the eigenvectors corresponding to $\lambda_j^{1/2}$ and $-\lambda_j^{1/2}$, respectively.
\end{proof}

\begin{remark}\label{R:measurable}
It is not difficult to prove versions of these results in the measurable setup. Let $(X,\mathcal{X},\mu)$ be a $\sigma$-finite measure space and $(\mathcal{E},\mathcal{F})$ a Dirichlet form on $L_2(X,\mu)$. Then the collection $\mathcal{B}:=\mathcal{F}\cap L_\infty(X,\mu)$ of (equivalence classes of) bounded energy finite functions on $X$ provides a (normed) algebra. If $(\mathcal{E},\mathcal{F})$ admits a carr\'e du champ, \cite[Chapter I]{BH91}, then we can use $\mathcal{B}$ and $\mathcal{B}^\ast$ in place of $\mathcal{C}$ and $\mathcal{C}^\ast$, respectively, to introduce the spaces $\mathcal{H}$ and $\mathbb{H}$ and the operators $\partial$, $\partial^\ast$ and $\mathbb{D}$ in a similar manner as before.
\end{remark}

\begin{examples}\label{Ex:Dirac}\mbox{}
We collect some examples that satisfy the hypotheses of this section.
\begin{enumerate}
\item[(i)] Consider the Sierpinski gasket $X=SG$, equipped with the resistance metric and the natural self-similar normalized Hausdorff measure $\mu$ and let $(\mathcal{E},\mathcal{F})$ be the local regular Dirichlet form on $L_2(SG,\mu)$, determined by the standard energy form on $SG$, cf. \cite[Theorems 3.4.6 and 3.4.7]{Kig01}. It is known that the generator of $(\mathcal{E},\mathcal{F})$ has discrete spectrum. This can also be observed for more general resistance forms on p.c.f. self-similar sets. See \cite[Theorems 2.4.1, 2.4.2, 3.4.6 and 3.4.7]{Kig01}.
\item[(ii)] Generators of local regular Dirichlet forms on generalized Sierpinski carpets, considered with the natural normalized Hausdorff measure, have pure point spectrum, see \cite{BB89, BBKT} and in particular \cite[Proposition 6.15]{BB99}.
\item[(iii)] Let $X=\mathbb{R}^n$ and let $(\mathcal{E},\mathcal{F})$ be the quadratic form associated with a Schr\"odinger operator $H=-\Delta+V$. Under some conditions on the potential $V$ (for instance continuity and nonnegativity) the associated form will be a Dirichlet form, and under further conditions on $V$ (for instance unboundedness at infinity) the operator $H$ will have discrete spectrum. See for example \cite[XIII.6.7-XIII.6.9]{RSIV}. This also applies to relativistic Schr\"odinger operators and, more generally, to Schr\"odinger operators associated with L\'evy processes, \cite{CMS90}.
\end{enumerate}
\end{examples}

\section[Spectral triples]{Spectral triples}\label{S:triples}

In this section we consider spectral triples associated with the Dirac operators $\mathbb{D}$ defined by formula (\ref{E:Diracdef}) in the preceding section.

Let $X$ be a locally compact separable metric space and $\mu$ a nonnegative Radon measure on $X$ with $\mu(U)>0$ for any nonempty open $U\subset X$. Let $(\mathcal{E},\mathcal{F})$ be a regular symmetric Dirichlet form on $L_2(X,\mu)$. As before we assume that $(\mathcal{E},\mathcal{F})$ has no killing part. As in the previous section we consider the natural complexifications of $\mathcal{E}$, $\mathcal{F}$, $\Gamma$ etc.

Since the kernel $ker\:\mathbb{D}$ may be infinite dimensional, we discuss a generalized notion of spectral triple similar to the one proposed in \cite[Definition 2.1]{CGIS12}.

\begin{definition}\label{D:triple}
A (possibly kernel degenerate) \emph{spectral triple} for an involutive algebra $A$ is a triple $(A,H,D)$ where $H$ is a Hilbert space and $(D, dom\:D)$ a self-adjoint operator on $H$  such that
\begin{enumerate}
\item[(i)] there is a faithful $\ast$-representation $\pi:A\to L(H)$,
\item[(ii)] there is a dense $\ast$-subalgebra $A_0$ of $A$ such that for all $a\in A_0$ the commutator $[D,\pi(a)]$ is well defined as a bounded linear operator on $H$,
\item[(iii)] the operator $(1+D)^{-1}$ is compact on $(ker\:D)^\bot$.
\end{enumerate} 
\end{definition}

If the reference measure $\mu$ is an energy dominant measure for $(\mathcal{E},\mathcal{F})$, i.e. if $(\mathcal{E},\mathcal{F})$ admits a carr\'e du champ, \cite{BH91}, then the space 
\begin{equation}\label{E:AA0}
\mathbb{A}_0:=\left\lbrace f\in\mathcal{C}:\Gamma(f)\in L_\infty(X,\mu)\right\rbrace
\end{equation}
is well defined and, according to Lemma \ref{L:coords},  $\mathcal{E}_1$-dense in $\mathcal{F}$. The Markov property of $(\mathcal{E},\mathcal{F})$ implies that $\mathbb{A}_0$ is an involutive algebra of functions, see Corollary \ref{C:multGamma} in the Appendix. Let $\mathbb{A}$ be the $C^\ast$-subalgebra of $C_0(X)$ obtained as the closure of $\mathbb{A}_0$,
\begin{equation}\label{E:fatA}
\mathbb{A}:=\clos_{C_0(X)}(\mathcal{A}_0). 
\end{equation}

\begin{remark}\mbox{}
\begin{enumerate}
\item[(i)] By definition any coordinate sequence for $(\mathcal{E},\mathcal{F})$ with respect to $\mu$ is contained in the algebra $\mathbb{A}_0$. The Stone-Weierstrass theorem implies that if there exists a point separating coordinate sequence $(f_n)_n$ that vanishes nowhere (i.e. such that for any $x\in X$ there exists some $f_n$ with $f_n(x)\neq 0$) then $\mathbb{A}$ agrees with the space $C_0(X)$.
\item[(ii)] If $m$ is a given nonnegative Radon measure on $X$ with $m(U)>0$ for any nonempty open $U\subset X$ and such that $\left\lbrace f\in\mathcal{C}: \Gamma(f)\in L_\infty(X,m)\right\rbrace$ is $\mathcal{E}$-dense in $\mathcal{F}$, then $m$ is energy dominant for $(\mathcal{E},\mathcal{F})$. This follows from (\ref{E:approxGamma}).
\end{enumerate}
\end{remark}

If in addition the generator $L$ of $(\mathcal{E},\mathcal{F})$ has discrete spectrum, i.e. if there exists a monotonically increasing sequence $0\leq \lambda_1\leq \lambda_2\leq ...$ of isolated eigenvalues $\lambda_j$ of $-L$ with finite multiplicity and $\lim_{j\to\infty}\lambda_j=+\infty$, together with an orthonormal basis $\left\lbrace \varphi_j\right\rbrace_j$ of corresponding eigenfunctions in $L_2(X,\mu)$, then the Dirac operator $\mathbb{D}$ on the Hilbert space $\mathbb{H}$ gives rise to a spectral triple for $\mathbb{A}$. 

\begin{theorem}\label{T:triple}
Let $(\mathcal{E},\mathcal{F})$ be a regular symmetric Dirichlet form on $L_2(X,\mu)$ and assume $\mu$ is energy dominant for $(\mathcal{E},\mathcal{F})$. Then we have the following.
\begin{itemize}
\item[(i)] There is a faithful representation $\pi:\mathbb{A}\to L(\mathbb{H})$,
\item[(ii)] For any $a\in\mathbb{A}_0$ the commutator $[\mathbb{D}, \pi(a)]$ is a bounded linear operator on $\mathbb{H}$,
\item[(iii)] If the $L_2(X,\mu)$-generator $L$ of $(\mathcal{E},\mathcal{F})$ has discrete spectrum then $(1+\mathbb{D})^{-1}$ is compact on $(ker\:\mathbb{D})^\bot$, and $(\mathbb{A},\mathbb{H},\mathbb{D})$ is a spectral triple for $\mathbb{A}$.
\end{itemize}
\end{theorem}

The proof of (ii) uses the fact that given an energy dominant measure $m$, the Hilbert space $\mathcal{H}$ can be written as the direct integral with respect to $m$ of a measurable field of Hilbert spaces $\left\lbrace \mathcal{H}_x\right\rbrace_{x\in X}$, cf. \cite{Dix, Tak}.

\begin{theorem}\label{T:measurablefield}
Let $(\mathcal{E},\mathcal{F})$ be a regular symmetric Dirichlet form on $L_2(X,\mu)$ and let $m$ be an energy dominant measure for $(\mathcal{E},\mathcal{F})$. Then there are a measurable field of Hilbert $\mathcal{C}$-modules on which the action of $a\in\mathcal{C}$ on $\omega_x\in\mathcal{H}_x$ is given by $a(x)\omega_x\in\mathcal{H}_x$ and such that the direct integral $\int_X^\oplus \mathcal{H}_x\:m(dx)$ is isometrically isomorphic to $\mathcal{H}$. In particular,
\[\left\langle \omega,\eta\right\rangle_{\mathcal{H}}=\int_X\left\langle \omega_x,\eta_x\right\rangle_{\mathcal{H}_x}\:m(dx)\]
for all $\omega,\eta\in\mathcal{H}$, where we write $\left\langle\cdot,\cdot\right\rangle_{\mathcal{H}_x}$ for the scalar products in the spaces $\mathcal{H}_x$, respectively. Given $f,g\in\mathcal{F}$, we have $\Gamma(f,g)(x)=\left\langle \partial_xf,\partial_xg\right\rangle_{\mathcal{H}_x}$ for $m$-a.e. $x\in X$, where $\partial_xf:=(\partial f)_x$.
\end{theorem}

Theorem \ref{T:measurablefield} follows by fixing $m$-versions of the functions $\Gamma(f,g)$, possible thanks to the separability of the Hilbert space $(\mathcal{F},\mathcal{E}_1)$. See \cite[Section 2]{HRT} for a proof. 

To prove Theorem \ref{T:triple} (ii) we also make use of a product rule for the operator $\partial^\ast$. Given $a\in\mathcal{C}$ and $u\in\mathcal{C}^\ast$, define their product $au\in\mathcal{C}^\ast$ by
\[(au)(\varphi):=u(a\varphi), \ \ \varphi\in\mathcal{C}.\]

For $\omega\in\mathcal{H}$ define a mapping from $\mathcal{H}$ into $\mathcal{C}^\ast$ (actually $C_0(X)^\ast$ would suffice) by 
\[(\omega^\ast \eta)(\varphi):=\left\langle \varphi\omega,\overline{\eta}\right\rangle_\mathcal{H}, \eta\in\mathcal{H}, \ \ \varphi\in\mathcal{C}.\]
If $m$ is an energy dominant measure,  $\left\lbrace \mathcal{H}_x\right\rbrace_{x\in X}$ is the corresponding measurable field of Hilbert spaces as in Theorem \ref{T:measurablefield}, and the function $x\mapsto \left\langle \omega_x,\overline{\eta}_x\right\rangle_{\mathcal{H}_x}$ is in $L_2(X,m)$, then $\omega^\ast \eta \in\mathcal{C}^\ast$ agrees with it and therefore may itself be seen as a function in $L_2(X,m)$. 

\begin{lemma}\label{L:productrulediv}
We have
\begin{equation}\label{E:productrulediv}
\partial^\ast(a\omega)=a\partial^\ast \omega-\omega^\ast\partial a
\end{equation}
for all $a\in\mathcal{C}$ and $\omega\in\mathcal{H}$, seen as an equality in $\mathcal{C}^\ast$. If the reference measure $\mu$ itself is energy dominant, $\omega\in dom\:\partial^\ast$ and $a\in\mathbb{A}_0$, then $a\omega\in dom\:\partial^\ast$, and (\ref{E:productrulediv}) holds in $L_2(X,\mu)$.
\end{lemma}

\begin{proof}
The validity of (\ref{E:productrulediv}) in $\mathcal{C}^\ast$ is a consequence of the Leibniz rule for $\partial$ together with the integration by parts identity (\ref{E:IbP}), see \cite{HRT}. If $\mu$ is energy dominant, $\omega\in dom\:\partial^\ast$ and $a\in\mathbb{A}_0$, then obviously $a\partial^\ast\omega\in L_2(X,\mu)$. But we also have $\omega^\ast\partial a\in L_2(X,\mu)$, because
\[\int_X|\left\langle \omega_x,\partial_x \overline{a}\right\rangle_{\mathcal{H}_x}|^2\mu(dx)\leq \int_X\left\|\omega_x\right\|_{\mathcal{H}_x}^2\left\|\partial_x a\right\|_{\mathcal{H}_x}^2\mu(dx)\leq \left\|\Gamma(a)\right\|_{L_\infty(X,\mu)}\left\|\omega\right\|_{\mathcal{H}}^2.\] 
This implies the lemma.
\end{proof}

We prove Theorem \ref{T:triple}.
\begin{proof}
Given $a\in\mathbb{A}_0$, let $\pi(a)$ denote the multiplication operator on $\mathbb{H}$, defined by
\[\pi(a)(f,\omega):=(af,a\omega),\ \ (f,\omega)\in\mathbb{H}.\]
Clearly $\pi(a)$ is bounded, and $\pi$ extends to a faithful representation of $\mathbb{A}$ on $\mathbb{H}$, what shows (i). To shorten notation we write again $a$ instead of $\pi(a)$. For (ii) note first that for any $a\in\mathbb{A}_0$ the commutator $[\mathbb{D},a]$ is well defined as linear operator from $\mathcal{C}\oplus \mathcal{H}\subset dom\:\mathbb{D}$ into $\mathcal{C}^\ast\oplus \mathcal{H}\supset\mathbb{H}$, and
by Lemma \ref{L:productrulediv} together with the Leibniz rule we have
\begin{align}
[\mathbb{D},a](f,\omega)&=\mathbb{D}(af, a\omega)-a\mathbb{D}(f,\omega)\notag\\
&=(\partial^\ast(a\omega),\partial(af))-(a\partial^\ast\omega,a\partial f)\notag\\
&=(-\omega^\ast\partial a,f\partial a)\notag
\end{align}
for any $(f,\omega)\in \mathcal{C}\oplus \mathcal{H}$. However, the norm bound
\begin{align}\label{E:normbound}
\left\|[\mathbb{D},a](f,\omega)\right\|_{\mathbb{H}}^2&=\int_X |\left\langle \omega_x,\partial_x \overline{a}\right\rangle_{\mathcal{H}_x}|^2 \mu(dx)+\int_X\left\| f(x)\partial_x a\right\|_{\mathcal{H}_x}^2\mu(dx)\notag\\
&\leq \left\|\Gamma(a)\right\|_{L_\infty(X,\mu)}\left\|(f,\omega)\right\|_{\mathbb{H}}^2.
\end{align}
shows that $[\mathbb{D},a](f,\omega)$ is a member of $\mathbb{H}$, and by the density of $\mathcal{C}\oplus \mathcal{H}$ in $\mathbb{H}$ the commutator $[\mathbb{D},a]$ extends to a bounded linear operator on $\mathbb{H}$. The operator $(1+\mathbb{D})^{-1}$ is compact because if $L$ has discrete spectrum with spectral representation (\ref{E:specL})
then $\mathbb{D}$ admits the spectral representation (\ref{E:Diracrep}) with $\lim_{j\to \infty} (\pm \lambda_j^{-1/2})=0$. 
\end{proof}

Now let $A$ be a point separating $C^\ast$-subalgebra of $C(X)$ and $(A,H,D)$ be a spectral triple for $A$. Let $A_0$ be a dense $\ast$-subalgebra of $A$ such that $[D,a]$ is bounded on $H$, cf. Definition \ref{D:triple} (ii). Then 
\[d_{D}(x,y):=\sup\left\lbrace a(x)-a(y)\ |\ a\in A_0 \ \text{ is such that }\left\|[D,a]\right\|\leq 1\right\rbrace\]
defines a metric in the wide sense on $X$, (a version of) the \emph{Connes distance}.

\begin{theorem}\label{T:Connesmetric}
Let $(\mathcal{E},\mathcal{F})$ be a strongly local Dirichlet form on $L_2(X,\mu)$, let $m$ be energy dominant for $(\mathcal{E},\mathcal{F})$ and assume there exists a point separating coordinate sequence for $(\mathcal{E},\mathcal{F})$ with respect to $m$. Let $\mathbb{D}$, $\mathbb{A}_0$ and $\mathbb{A}$ be given as in (\ref{E:Diracdef}), (\ref{E:AA0}) and (\ref{E:fatA}), respectively. Then 
\[d_{\mathbb{D}}(x,y):=\sup\left\lbrace a(x)-a(y)\ |\ a\in \mathbb{A}_0 \text{ is such that }\left\|[\mathbb{D},a]\right\|\leq 1\right\rbrace\]
is a metric in the wide sense on $X$ and $d_{\mathbb{D}}\leq d_{\Gamma,\mu}$. If $X$ is compact then $d_{\mathbb{D}}= d_{\Gamma,\mu}$.
\end{theorem}

\begin{proof} It is obvious that $d_\mathbb{D}$ is a metric in the wide sense. We next first verify that 
\begin{equation}\label{E:claim}
\left\|[\mathbb{D},a]\right\|^2=\left\|\Gamma(a)\right\|_{L_\infty(X,\mu)}
\end{equation}
for any $a\in\mathbb{A}_0$. By (\ref{E:normbound}) we have $\left\|[\mathbb{D},a]\right\|^2\leq\left\|\Gamma(a)\right\|_{L_\infty(X,\mu)}$. Now assume that $\lambda:=\left\|[\mathbb{D},a]\right\|^2< \left\|\Gamma(a)\right\|_{L_\infty(X,\mu)}$. Then we could find some some Borel set $A\subset X$ and some $\delta>0$ such that $0<\mu(A)<+\infty$ and $\left\|\partial_xa\right\|_{\mathcal{H}_x}^2=\Gamma_x(A)>\lambda(1+\delta)$ for all $x\in A$. Since $\mu$ is Radon, there would be some compact set $K\subset A$ with $\mu(K)>0$ and some open set $U\supset K$ with
\begin{equation}\label{E:measurerelation}
\mu(U\setminus K)< \delta\:\mu(K).
\end{equation} 
Let $f\in\mathcal{C}$ be a function supported in $U$ such that $0\leq f\leq 1$ and $f(x)=1$ for $x\in K$. By the regularity of $(\mathcal{E},\mathcal{F})$ such $f$ exists, cf. \cite[Problem 1.4.1]{FOT94}. Then, according to (\ref{E:measurerelation}), 
\[\delta \int_K f(x)^2\mu(dx)>\int_{X\setminus K}f(x)^2\mu(dx)\]
and therefore  
\[\left\|f\partial a\right\|_\mathcal{H}^2=\int_X f(x)^2\left\|\partial_xa\right\|_{\mathcal{H}_x}^2\mu(dx)>\lambda (1+\delta)\int_K f(x)^2\mu(dx)>\lambda\left\|f\right\|_{L_2(X,\mu)}^2.\]
This would imply $\left\|[\mathbb{D},a](f,0)\right\|_\mathbb{H}>\lambda \left\|(f,0)\right\|_\mathbb{H}$, a contradiction. Therefore (\ref{E:claim}) holds. Since $\mathbb{A}_0\subset\mathcal{F}_{loc}\cap C(X)$ we have $d_\mathbb{D}\leq d_{\Gamma,\mu}$.  If $X$ is compact, then $\mathcal{F}_{loc}\cap C(X)=\mathcal{C}$, hence $\mathcal{A}$ as defined in (\ref{E:A}) coincides with $\mathbb{A}_0$ and consequently $d_\mathbb{D}=d_{\Gamma,\mu}$. 
\end{proof}

\begin{remark}
Either of the following conditions imply the equality $d_{\mathbb{D}}= d_{\Gamma,\mu}$ also for noncompact $X$:
\begin{enumerate}
\item[(i)] The distance $d_{\Gamma,\mu}$ induces the original topology and $(X,d_{\Gamma,\mu})$ is complete.
\item[(ii)] The distance $d_{\mathbb{D}}$ induces the original topology and $(X,d_{\mathbb{D}})$ is complete. 
\item[(iii)] For any relatively compact open set $U\subset X$ there exists a function $\varphi\in\mathcal{C}$ with $0\leq \varphi\leq 1$ and $\varphi(x)=1$ for $x\in U$ such that $\Gamma(\varphi)\leq m$.  
\end{enumerate}
Both (i) and (ii) imply the desired equality by \cite[Appendix 4.2, Proposition 1 (c)]{Stu94} together with (\ref{E:claim}). 
Condition (iii) allows a suitable cut-off argument. Note that it is always possible to construct a finite energy dominant measure for which (iii) is valid: If $X=\bigcup_n U_n$ is an exhaustion of $X$ by an increasing sequence of relatively compact open sets $U_n$ with $\overline{U_n}\subset U_{n+1}$ then there are functions $\varphi_n\in\mathcal{C}$ such that $0\leq \varphi_n\leq 1$, $\varphi_n(x)=1$ for $x\in U_n$ and $\varphi_n(x)=0$ for $x\in U_{n+1}^c$. It suffices to adjoin the countable collection $\left\lbrace \varphi_n\right\rbrace_n$ to the functions in the construction of the measure $m_0$ in Lemma \ref{L:coords}.
\end{remark}

\begin{examples}\mbox{}
\begin{enumerate}
\item[(i)] Consider the Sierpinski gasket $X=SG$, equipped with the resistance metric. From the standard energy form we can construct the  Kusuoka measure $\nu$, \cite{Ka12, Kig08, KZh, Ku89, T08}: There is a complete (up to constants) energy orthonormal system $\left\lbrace h_1, h_2\right\rbrace$ of harmonic functions on $SG$, and $\nu$ is defined as the sum of their energy measures, $\nu:=\Gamma(h_1)+\Gamma(h_2)$.
The Kusuoka measure is energy dominant. Note that the self-similar normalized Hausdorff measure is not energy dominant, \cite{BBST99}.
Let $(\mathcal{E},\mathcal{F})$ be the local regular Dirichlet form on $L_2(SG,\nu)$ induced by the standard energy form on $SG$. Its generator has discrete spectrum, cf. \cite[Theorems 2.4.1, 2.4.2, 3.4.6 and 3.4.7]{Kig01}. Moreover, all functions of finite energy are continuous, and we have $\mathbb{A}=C(SG)$ in (\ref{E:fatA}). Theorem \ref{T:triple} yields a spectral triple $(C(SG),\mathbb{H},\mathbb{D})$
for $C(SG)$.

In a similar manner we can obtain spectral triples associated with regular resistance forms on finitely ramified fractals, equipped with the Kusuoka measure, \cite{Kig03, T08}. Note that any nonatomic Borel measure (with respect to the resistance metric) satisfying some growth condition turns the given resistance form into a Dirichlet form having a generator with discrete spectrum, see \cite[Theorem 8.10, Proposition 8.11 and the remark following the proof of Lemma 8.12]{Kig03}.
\item[(ii)] For the standard self-similar Dirichlet form on the $2$-dimensional Sierpinski carpet, equipped with the natural self-similar normalized Hausdorff measure $\mu$ the energy measures are singular with respect to $\mu$, see \cite{Hino03, Hino05, Hino08}. It is always possible to construct energy dominant measures, \cite{Hino08, HRT}, and under some conditions one can establish an energy dominant measure as new reference measure by means of time change, see \cite[Section 6.2]{FOT94} and the references therein.
\item[(iii)] For the Dirichlet forms associated with Schr\"o\-din\-ger operators and relativistic Schr\"o\-din\-ger operators mentioned in Examples \ref{Ex:Dirac} (iii) the Lebesgue measure is energy dominant, and Theorem \ref{T:triple} yields associated spectral triples.
\item[(iv)] If a purely nonlocal regular Dirichlet form has a jump kernel with respect to the product $\mu\otimes\mu$ of the reference measure $\mu$, then $\mu$ is energy dominant. Conditions for general purely nonlocal Dirichlet forms to have discrete spectrum are provided in \cite[Section 4, in particular Corollary 4.2]{Wa00}. In some cases, for instance when using subordination, \cite{JSch99, K02}, discrete spectrum may be observed directly.
\end{enumerate}
\end{examples}

\begin{remark}
If $\mu'$ is a nonnegative Radon measure on $X$ such that $\mu$ and $\mu'$ are mutually absolutely continuous and $\mu\leq \mu'$, then $(\mathcal{E},\mathcal{C})$ is closable in $L_2(X,\mu')\subset L_2(X,\mu)$ and its closure $(\mathcal{E},\mathcal{F}')$ is a regular Dirichlet form with $\mathcal{F}'\subset \mathcal{F}$.
If $\mathcal{H}'$ defines the Hilbert space $\mathcal{H}$ defined with $\mu'$ in place of $\mu$, and similarly for the other objects, we observe
 $\mathcal{H}'=\mathcal{H}$,  $\mathbb{H}'\subset \mathbb{H}$ and $\mathbb{A}_0'=\mathbb{A}_0$. The operator $\mathbb{D}$ is an extension of (a restriction of) $\mathbb{D}'$. For the corresponding distances we have $d_\mathbb{D}\leq d_{\mathbb{D}'}$.
\end{remark}

\section[Metrics and gradient fields]{Metrics and gradient fields}\label{S:gradient}

The intrinsic metric $d_{\Gamma,m}$ of a strongly local Dirichlet form with respect to an energy dominant measure $m$ can also be expressed in terms of vector fields. As in Section \ref{S:Dirichlet} let $X$ be a locally compact separable metric space, $\mu$ a nonnegative Radon measure on $X$ such that $\mu(U)>0$ for all nonempty open $U\subset X$ and $(\mathcal{E},\mathcal{F})$ a strongly local Dirichlet form on $L_2(X,\mu)$.
Let $B_b(X)$ denote the space of bounded Borel functions on $X$. We consider the tensor product $\mathcal{F}_{loc}\cap C(X)\otimes \mathcal{B}_b(X)$. For any compact set $K\subset X$ we can define a symmetric bilinear form on $\mathcal{F}_{loc}\cap C(X)\otimes \mathcal{B}_b(X)$ by 
\begin{equation}\label{E:norm}
\left\langle \sum_i f_i\otimes g_i,\sum_i f_i\otimes g_i\right\rangle_{\mathcal{H}(K)}:=\sum_i\sum_j\int_K g_ig_jd\Gamma(f_i,f_j).
\end{equation}
This form is nonnegative definite, as may be seen using step functions in place of the $g_i$. Its square root defines a seminorm $\left\|\cdot\right\|_{\mathcal{H}(K)}$ on $\mathcal{F}_{loc}\cap C(X)\otimes \mathcal{B}_b(X)$, and by $\mathcal{H}(K)$ we denote the Hilbert space obtained by factoring out elements of zero seminorm and completing. Similarly as before we define a right action of $\mathcal{B}_b(X)$ on $\mathcal{F}_{loc}\cap C(X)\otimes \mathcal{B}_b(X)$ by 
\begin{equation}\label{E:right}
(f\otimes g)h:=f\otimes (gh).
\end{equation}
For any $K\subset X$ compact, $\left\|(\sum_i f_i\otimes g_i)h\right\|_{\mathcal{H}(K)}\leq \sup_{x\in K}|h(x)|\left\|\sum_i f_i\otimes g_i\right\|_{\mathcal{H}(K)}$, hence (\ref{E:right}) extends to an action of $\mathcal{B}_b(X)$ that is bounded on $\mathcal{H}(K)$. For any finite linear combination $\sum_i f_i\otimes g_i$ from $\mathcal{F}_{loc}\cap C(X)\otimes \mathcal{B}_b(X)$ the integrand 
\[\sum_i\sum_j g_ig_jd\Gamma(f_i,f_j)\]
on the right hand side of (\ref{E:norm}) defines a nonnegative measure on $X$. Consequently we have $\left\|\cdot\right\|_{\mathcal{H}(K)}\leq \left\|\cdot\right\|_{\mathcal{H}(K')}$ for any two compact sets $K, K'\subset X$ with $K\subset K'$. This implies that the restriction $v\mathbf{1}_K$ to $K$ in the sense of (\ref{E:right}) of any $v\in\mathcal{H}(K')$ is a well defined element of $\mathcal{H}(K)$. Together with this restriction the spaces $\mathcal{H}(K)$, $K\subset X$ compact, form an inverse system of Hilbert spaces, and we denote its inverse limit by $\mathcal{H}_{loc}$. If $(K_n)_n$ is an exhaustion of $X$ by compact sets $K_n$ then the family of seminorms $\left\|\cdot\right\|_{\mathcal{H}(K_n)}$, $n\in\mathbb{N}$, induces the topology of $\mathcal{H}_{loc}$. The space $\mathcal{H}_{loc}$ is locally convex. 
A left action of $\mathcal{F}_{loc}\cap C(X)$ on $\mathcal{F}_{loc}\cap C(X)\otimes \mathcal{B}_b(X)$ can be defined by
\begin{equation}\label{E:left}
h(f\otimes g):=(fh)\otimes g-h\otimes (fg).
\end{equation}
Corollary \ref{C:multGamma} and the nonnegativity of the measure on the right hand side of (\ref{E:norm}) imply
$\left\|h(\sum_i f_i\otimes g_i)\right\|_{\mathcal{H}(K)}\leq \sup_{x\in K}|h(x)|\left\|\sum_i f_i\otimes g_i\right\|_{\mathcal{H}(K)}$,
hence also (\ref{E:left}) extends to a bounded action on each $\mathcal{H}(K)$. The definition $\partial f:=f\otimes \mathbf{1}$ now provides a linear operator $\partial:\mathcal{F}_{loc}\cap C(X)\mapsto\mathcal{H}_{loc}$ such that for any $K\subset X$ the operator $\partial$ acts as a bounded derivation, more precisely, $\left\|\partial f\right\|_{\mathcal{H}(K)}^2=\Gamma(f)(K)$
and $\partial (fg)=f\partial g+g\partial f$.

Now let $m$ be an energy dominant measure for $(\mathcal{E},\mathcal{F})$. By Theorem \ref{T:measurablefield} there exists a measurable field of Hilbert modules $\left\lbrace\mathcal{H}_x\right\rbrace_{x\in X}$ such that for any $K\subset X$ compact the space $\mathcal{H}(K)$ is isometrically isomorphic to the direct integral $\int_K^\oplus \mathcal{H}_x m(dx)$, in particular
\[\left\|v\right\|_{\mathcal{H}(K)}=\int_K \left\|v_x\right\|_{\mathcal{H}_x}^2 m(dx)\]
for any measurable section $v=(v_x)_{x\in X}$ of $\left\lbrace\mathcal{H}_x\right\rbrace_{x\in X}$. Let $L_\infty(X,m,\left\lbrace\mathcal{H}_x\right\rbrace)$ denote the space of $m$-equivalence classes of measurable sections $v=(v_x)_{x\in X}$ such that \[\left\|v\right\|_{L_\infty(X,m,\left\lbrace\mathcal{H}_x\right\rbrace)}:=\esssup_{x\in X} \left\|v_x\right\|_{\mathcal{H}_x}\] 
is finite. The space $L_\infty(X,m,\left\lbrace\mathcal{H}_x\right\rbrace)$ is a Banach space, as can be seen using a version of the classical Riesz-Fischer type argument. Since the measure $m$ is Radon, we have $L_\infty(X,m,\left\lbrace\mathcal{H}_x\right\rbrace)\subset \mathcal{H}_{loc}$. We refer to $L_\infty(X,m,\left\lbrace\mathcal{H}_x\right\rbrace)$ as the \emph{space of bounded vector fields}. It allows a natural description of the intrinsic metric in terms of functions with gradient fields that are $L_\infty$-bounded by one.

\begin{theorem}
Let $(\mathcal{E},\mathcal{F})$ be a strongly local Dirichlet form on $L_2(X,\mu)$ and let $m$ be an energy dominant measure for $(\mathcal{E},\mathcal{F})$. We have 
\[d_{\Gamma,m}(x,y)=\sup\left\lbrace f(x)-f(y)\ |\  f\in\mathcal{F}_{loc}\cap C(X)\text{ is such that } \left\|\partial f\right\|_{L_\infty(X,m,\left\lbrace\mathcal{H}_x\right\rbrace)}\leq 1\right\rbrace.\]
for all $x,y\in X$.
\end{theorem}

\section[Appendix]{Appendix}\label{S:App}

The following statements are versions of results on composition and multiplication from \cite[Chapter I]{BH91}. Let $X$ be a locally compact separable metric space, $\mu$ a nonnegative Radon measure on $X$ such that $\mu(U)>0$ for all nonempty open $U\subset X$ and let $(\mathcal{E},\mathcal{F})$ be a regular symmetric Dirichlet form on $L_2(X,\mu)$.

Given $n\in\mathbb{N}\setminus\left\lbrace 0\right\rbrace$ let $\mathcal{T}_n^0$ denote the set of all normal contractions, that is, functions $F:\mathbb{R}^n\to\mathbb{R}$ such that $F(0)=0$ and $|F(x)-F(y)|\leq \sum_{i=1}^n|x_i-y_i|,\ \ x,y\in\mathbb{R}^n$.

\begin{lemma}\label{L:composeGamma}
Let $f_1,...,f_n\in\mathcal{C}$, $F\in\mathcal{T}_n^0$ and $g:=F(f_1,...,f_n)$. Then we have
\[(\mathcal{E}(gh,g)-\frac12\mathcal{E}(g^2h))^{1/2}\leq\sum_{i=1}^n (\mathcal{E}(f_ih,f_i)-\frac12\mathcal{E}(f_i^2,h))^{1/2}\]
for all nonnegative $h\in\mathcal{C}$.
\end{lemma}

The lemma can by proved by arguments similar to those used for \cite[Propositions I.2.3.3 and I.3.3.1]{BH91}. The following corollary is an immediate consequence, cf. \cite[Corollary I.3.3.2]{BH91}.

\begin{corollary}\label{C:multGamma}
For any $f,g\in\mathcal{C}$ and any Borel set $A\subset X$ we have 
\[\Gamma(fg)(A)^{1/2}\leq \sup_{x\in A}|f(x)|\Gamma(g)(A)^{1/2}+\sup_{x\in A}|g(x)|\Gamma(f)(A)^{1/2}.\]
If $m$ is energy dominant, then in particular
\[\Gamma(fg)(x)\leq 2\left(\Gamma(f)(x)\left\|g\right\|^2_{L_\infty(X,m)}+\Gamma(g)(x)\left\|f\right\|_{L_\infty(X,m)}^2\right)\]
for $m$-a.e. $x\in X$.
\end{corollary}

\end{document}